\documentclass{article}

\usepackage{amsmath}
\usepackage{amssymb}
\usepackage{amsthm}
\usepackage[a4paper]{geometry}
\usepackage{graphicx}
\usepackage{tikz}
\usetikzlibrary{shapes.geometric,arrows}
\usepackage{setspace}
\usepackage{enumerate}
\numberwithin{equation}{section}
\newtheorem{theorem}{Theorem}[section]
\newtheorem{lemma}[theorem]{Lemma}
\newtheorem{proposition}[theorem]{Proposition}

\newtheorem{definition}[theorem]{Definition}
\newtheorem{example}[theorem]{Example}

\newtheorem{note}[theorem]{Note}

\title{Ideal on CL-algebra}
\author{Safiqul Islam, Arundhati Sanyal and Jayanta Sen\\
	Department of Mathematics, Taki Government College, West Bengal, India\\
	safiqulwbes@gmail.com, arundhatisanyal.byasrki@gmail.com, jayanta.wbes@gmail.com}

\begin{document}
	
	\maketitle

\begin{abstract}
In this paper, we introduce the concept of filter on IL-algebra. It is proved that this concept generalizes the notion of filter on Residuated Lattices. Prime filters on IL-algebra are defined and few interesting properties are obtained. It has been shown that quotient algebra corresponding to IL-algebra is formed with the help of filters also an IL-algebra.

\end{abstract}

	
\section{Introduction}
Filter theory plays important role in the study of algebraic structures and associated logics, in some cases\cite{ras}. Recently, various researchers work on filters of various algebraic structures. P. H\'ajek introduced the concept of BL-algebra and also introduced the concept of filter on BL-algebra\cite{haj}. After that many authors contributed on various types of filters of BL-algebra\cite{tur1,tur2,sae1,sae2,sae3,sae4,mog,mot,kon}. Filters on other algebraic structures are also available in literature\cite{bor}. 

Intuitionistic Linear Algebra (IL-algebra, in short) was introduced by A. Troelstra\cite{tro} as an algebraic counterpart of linear logic\cite{gir}. IL-algebra can be found as $FL_e$-algebra in \cite{ono}. Various properties of IL-algebra are studied in \cite{senth,sen}. Filters on IL-algebra are not available in literature. In this paper, we introduce the concept of filters on IL-algebra. Properties of filters on IL-algebra are studied here.

In the next section, definition of IL-algebra is given. Properties and examples of IL-algebra are also discussed in this section. In section 3, filter on IL-algebra is introduced. Properties of filters on IL-algebra are obtained in this section. With the help of filter, we define a congruence relation on an arbitrary IL-algebra and prove that the corresponding quotient algebra is also an IL-algebra with respect to suitable operations. Some special type of filters on IL-algebra are discussed in Section 4.

\section{Intuitionistic Linear Algebra}
\begin{definition}
Let $L$ be a non empty set. An intuitionistic linear algebra (IL-algebra, in short)\cite{tro} is an algebraic system $L=(L,\cup ,\cap, \bot, \to, \ast,1 )$  which satisfies the following conditions:
\begin{itemize}
\item $(L,\cup,\cap,\bot)$ is a lattice with least element $\bot$.
\item $(L, \ast, 1)$ is a commutative monoid with unit 1.
\item for any $x,y,z \in L$ , $x\ast y \leq z$ if and only if $x \leq y \to z$. [residuation property]
\end{itemize}

\end{definition}

\begin{theorem}\label{lem1}
	In every IL-algebra $L$, the following results hold for all $x, y, x_1, y_1, z \in L$. \cite{tro,senth,sen}
	\begin{enumerate}
		\item $x \ast (y \cup z)= (x \ast y)\cup (x \ast z)$ and moreover, if the join $\underset{i\in I}{\bigcup}y_i$ exists, then $x \ast \underset{i\in I}{\bigcup}y_i= \underset{i\in I}{\bigcup}(x \ast y_i)$. 
		\item $\bot \to \bot$ is the largest element of $L$ and is denoted by $\top$.
		\item If $x, y \leq 1$  then $x \ast y  \leq x\cap y$.
		\item $1 \leq x, y$ then $x \cup y \leq x  \ast y$.
		\item $(x\rightarrow y)\ast (y\rightarrow z)\leq (x\rightarrow z)$.
		\item $1\to x =x$.
		\item If $x\leq x_1, y\leq y_1$ then $x\ast y\leq x_1\ast y_1$ and $x_1\to y\leq x\to y_1$.
		\item $x\to (y\to z) = (x \ast y)\to z$.
		\item $x\ast (x\to y) \leq y$.
		\item $1\leq x\to x$.
	\end{enumerate}
\end{theorem}

\begin{example} \label{ex1}
	Example of an IL-algebra.
	Let, $X=\{ \bot,b,c,d,1,\top \}$. Lattice ordering, $\ast$ and $\rightarrow$ tables are the following	
	
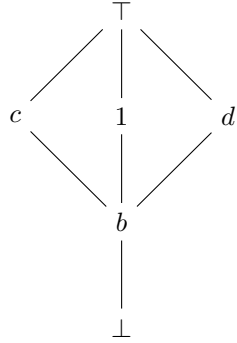
\begin{center}
		\begin{tikzpicture}[scale=.7]
		\node (top) at (0,2) {$\top$}; 
		\node (d) at (2,0) {$d$}; 
		\node (one) at (0,0) {$1$}; 
		\node (c) at (-2,0) {$c$};
		\node (b) at (0,-2) {$b$}; 
		\node (bot) at (0,-4) {$\bot$};
		\draw (bot) -- (b) -- (d) -- (top);
		\draw (b) -- (one) -- (top);
		\draw (b) -- (c) -- (top);
		\end{tikzpicture}
	\end{center}
\begin{center}
			\begin{tabular}{ l | l l l l l l }
			\textbf{$\ast$} & \textbf{$\bot$} & \textbf{$b$} & \textbf{$c$} & \textbf{$d$} & \textbf{$1$} & \textbf{$\top$}  \\ \hline
			$\bot$ & $\bot$ & $\bot$ & $\bot$ & $\bot$ & $\bot$ & $\bot$ \\ 
			$b$ & $\bot$ & $\bot$ & $b$ & $b$ & $b$ & $b$ \\ 
			$c$ & $\bot$ & $b$ & $c$ & $c$ & $c$ & $c$ \\ 
			$d$ & $\bot$ & $b$ & $c$ & $1$ & $d$ & $\top$ \\ 
			$1$ & $\bot$ & $b$ & $c$ & $d$ & $1$ & $\top$ \\ 
			$\top$ & $\bot$ & $b$ & $c$ & $\top$ & $\top$ & $\top$
		\end{tabular}		
	\hspace{1cm}
			\begin{tabular}{ l | l l l l l l }
			\textbf{$\rightarrow$} & \textbf{$\bot$} & \textbf{$b$} & \textbf{$c$} & \textbf{$d$} & \textbf{$1$} & \textbf{$\top$}  \\ \hline
			$\bot$ & $\top$ & $\top$ & $\top$ & $\top$ & $\top$ & $\top$ \\ 
			$b$ & $b$ & $\top$ & $\top$ & $\top$ & $\top$ & $\top$ \\
			$c$ & $\bot$ & $b$ & $\top$ & $b$ & $b$ & $\top$ \\ 
			$d$ & $\bot$ & $b$ & $c$ & $1$ & $d$ & $1$ \\ 
			$1$ & $\bot$ & $b$ & $c$ & $d$ & $1$ & $\top$ \\ 
			$\top$ & $\bot$ & $b$ & $c$ & $b$ & $b$ & $\top$
		\end{tabular}
	\end{center}		
	
	Then, $(L,\ast,\cup,\cap,\rightarrow,1,\top)$ is an IL-algebra.
	
\end{example}
\begin{example}\label{ex2}
	Example of an IL-algebra where strict inequality of the property $x\cup y \leq x\ast y$, for $1\leq x,y$ holds. \\
	Let, $L=\{\bot,b,c,d,1,\top\}$. Lattice ordering, $*$ and $\rightarrow$ tables are the following \\
	
	\begin{center}
		\begin{tikzpicture}[scale=1]
		\node (top) at (0,5) {$\top$}; 
		\node (d) at (0,4) {$d$}; 
		\node (c) at (0,3) {$c$}; 
		\node (b) at (0,2) {$b$};
		\node (one) at (0,1) {$1$}; 
		\node (bot) at (0,0) {$\bot$};
		\draw (bot) -- (one) -- (b) -- (c)-- (d) -- (top);
		\end{tikzpicture}
	\end{center}
\begin{center}
		\begin{tabular}{ l | l l l l l l }
			\textbf{$\ast$} & \textbf{$\bot$} & \textbf{$b$} & \textbf{$c$} & \textbf{$d$} & \textbf{$1$} & \textbf{$\top$}  \\ \hline
			$\bot$ & $\bot$ & $\bot$ & $\bot$ & $\bot$ & $\bot$ & $\bot$ \\ 
			$b$ & $\bot$ & $\top$ & $\top$ & $\top$ & $b$ & $\top$ \\ 
			$c$ & $\bot$ & $\top$ & $\top$ & $\top$ & $c$ & $\top$ \\ 
			$d$ & $\bot$ & $\top$ & $\top$ & $\top$ & $d$ & $\top$ \\ 
			$1$ & $\bot$ & $b$ & $c$ & $d$ & $1$ & $\top$ \\ 
			$\top$ & $\bot$ & $\top$ & $\top$ & $\top$ & $\top$ & $\top$
		\end{tabular} 
		\hspace{0.5cm}	
		\begin{tabular}{ l | l l l l l l }
			\textbf{$\rightarrow$} & \textbf{$\bot$} & \textbf{$b$} & \textbf{$c$} & \textbf{$d$} & \textbf{$1$} & \textbf{$\top$}  \\ \hline
			$\bot$ & $\top$ & $\top$ & $\top$ & $\top$ & $\top$ & $\top$ \\ 
			$b$ & $\bot$ & $1$ & $1$ & $1$ & $\bot$ & $\top$ \\
			$c$ & $\bot$ & $\bot$ & $1$ & $1$ & $\bot$ & $\top$ \\ 
			$d$ & $\bot$ & $\bot$ & $\bot$ & $1$ & $\bot$ & $\top$ \\ 
			$1$ & $\bot$ & $b$ & $c$ & $d$ & $1$ & $\top$ \\ 
			$\top$ & $\bot$ & $\bot$ & $\bot$ & $\bot$ & $\bot$ & $\top$ 
		\end{tabular}

\end{center}
	In this example, we can see that $c\ast d=\top$ but $c\cup d=d$.
\end{example}

\section{Filters}
In the second example of the previous section, we noticed that there are $x,y\geq 1$ for which $x\cap y \lneqq x\ast y$ holds. It motivates us to define filter on IL-algebra in the following manner.

\begin{definition}
Let $L$ be an IL-algebra. A non-empty subset $F$ of $L$ is said to be a filter if the following are satisfied
\begin{enumerate}
	\item $1 \in F$
	\item If $x, y \in F $ then $x \ast y \in F$ and $x \cap y \in F$
	\item If $x \in F$ and $x \leq y$ then $y \in F$
\end{enumerate}
\end{definition}

\begin{note} It may be noted that usually $x\cap y \in F$, for $x,y\in F$ is not taken in the definition of filter for other known algebraic structures like Residuated Lattice, BL-algebra. In these algebraic structures $x\ast y\leq x\cap y$ for all $x,y$ and so $x\ast y\in F$ implies $x\cap y\in F$. As the property $x\ast y\leq x\cap y$ for all $x,y$ is not available in IL-algebra, we have to take it in the definition of filter for IL-algebra.
\end{note}

\begin{example}
	Let $F=\{1,b,c,d,\top\}$ in Example \ref{ex2}. Then $F$ is a filter in $L$.
\end{example}
\begin{proposition}\label{prop1}
	Let, $F$ be a filter on an IL-algebra $L$. If $x \leq y$ then $x \to y \in F$.
\end{proposition}
\begin{proof}
	$x\leq y$ implies $1\leq (x\rightarrow y)$. \\
	So, $x\rightarrow y\in F$.
\end{proof}

Let $F$ be a filter of an IL-algebra $L$. We define a binary relation $\rho$ on $L$ by $x \rho y$ holds if and only if $x \to y \in F$ and $y \to x \in F$. It can be checked that $\rho$ is a congruence relation on $L$. The set of all congruence  classes is denoted by $ L/F $ i.e.,   \\
$L/ F = {\{[x]: x \in L\}} $ where $[x]= \{y \in L : x \rho y \}$. \\
Now we define $\cap, \cup, \to, \ast$, on $L/F$ by $[x] \Box [y]=[x \Box y]$ where $\Box \in \{\cap, \cup, \to, \ast\}$. \\
Now, $[\bot] =\{x \in L : x \to \bot \in F, \bot \to x \in F \} = \{x\in L: x\to \bot \in F \}$ by Proposition \ref{prop1}.  \\ 
Similarly, by Proposition \ref{prop1}, $[\top]=\{x\in L :\top \to x \in F \}$ and by Proposition \ref{prop1} and Theorem \ref{lem1}.6, we get $[1]=\{x\in F: x\to 1 \in F \}$.

\begin{proposition}\label{prop2}
	$[x]\leq [y]$ if and only if $x\to y \in F$.
\end{proposition}
\begin{proof}
		Let $[x]\leq [y]$. So, $[x]\cap [y] =[x]$. Thus $x\to (x\cap y)\in F$. As $x\cap y \leq y,\ (x\cap y)\to y \in F$.\\
	Then $(x\to (x\cap y)) \ast ((x\cap y)\to y )\in F$.\\
	From Theorem \ref{lem1}.5, we get $(x\to (x\cap y)) \ast ((x\cap y)\to y )\leq x\to y$.
	Therefore $x\to y \in F$.
	
	Conversely, let $x\to y \in F$. We want to show that $[x\cap y]=[x]$.\\
	Since $x\cap y\leq x$ therefore $x\cap y \to x \in F$. \\
	Again $1\leq x\to x$ and $x\to y \in F$ so $(x\to x)\cap (x \to y)\in F$.\\
	$x\ast ((x\to x)\cap (x \to y))\leq x\ast (x\to x) \leq x$ and\\
	$x\ast ((x\to x)\cap (x \to y))\leq x\ast (x\to y) \leq y$. \\
	Thus $x\ast ((x\to x)\cap (x \to y))\leq x\cap y$.\\
	Hence by residuation property, $(x\to x)\cap (x \to y) \leq x\to x\cap y$. \\
	Therefore $x\to (x \cap y)\in F$. 	So, $[x]\leq [y]$. 
\end{proof}
\begin{lemma}\label{lem2}
In an IL-algebra $L$, $(z\to x) \cap (z \to y)= z\to (x \cap y)$ for all $x,y,z \in L$.
\end{lemma}
\begin{proof}
	As $x \cap y \leq x,y$, by Theorem \ref{lem1}.7, we get $z\to (x \cap y)\leq z \to x, z\to y$. \\ 
	Therefore $z\to (x\cap y)\leq (z\to x)\cap(z\to y)$. \\ 
	Now $z\ast((z \to x)\cap(z \to y)) \leq z\ast (z\to x)\leq x $. \\
	Similarly, $z\ast((z \to x)\cap(z \to y)) \leq y $. \\
	So $z\ast ((z \to x)\cap (z \to y)) \leq x \cap y$. \\
	Then by residuation property, we get $(z \to x)\cap (z \to y) \leq z\to  (x \cap y)$. \\	
	Hence $(z \to x)\cap (z \to y) = z\to  (x \cap y)$. 	
\end{proof}

\begin{lemma}\label{lem3}
In an IL-algebra $L$, $(x\to z) \cap (y \to z)= (x\cup y)\to z$ for all $x,y,z \in L$.
\end{lemma}
\begin{proof}
	As $x,y\leq x\cup y$, by Theorem \ref{lem1}.7, we get $(x\cup y)\to z\leq x \to z, y\to z$. \\ 
	Therefore $(x\cup y)\to z\leq (x\to z)\cap(y\to z)$. \\ 
	Now $x\ast((x \to z)\cap(y \to z)) \leq x\ast (x\to z)\leq z $. \\
	Similarly, $y\ast((x \to z)\cap(y \to z)) \leq z $. \\
	So $(x\ast((x \to z)\cap(y \to z)))\cup (y\ast((x \to z)\cap(y \to z)))\leq z$. \\
	By Theorem \ref{lem1}.1, $((x \to z)\cap(y \to z))\ast (x\cup y)\leq z$
	Then by residuation property, we get $(x \to z)\cap(y \to z)\leq (x\cup y)\to z$. \\	
	Hence $(x\to z) \cap (y \to z)= (x\cup y)\to z$.	
\end{proof}

\begin{proposition}
	$L/F$ is closed under $\Box$, where $\Box \in \{\cap, \cup, \to, \ast\}$.
\end{proposition}

\begin{proof} 
\begin{itemize}
    \item We first show that $L/F$ is closed under $\cap$.\\
	Let $x_1\in[x] , y_1\in[y]$. \\
	We shall show that $[x_1]\cap [y_1]= [x] \cap [y]$ i.e., $[x_1\cap y_1] = [x\cap y]$.  Thus we have to show that
	$ (x \cap y) \to (x_1 \cap y_1) $ and $(x_1 \cap y_1) \to (x \cap y) $  both are in $F$.  \\
	By assumption $x \to x_1, x_1 \to x, y \to y_1, y_1\to y \in F$. \\ 
	Now by Theorem \ref{lem1}.7, $x\to x_1 \leq (x \cap y) \to x_1$. Then $(x\cap y) \to x_1 \in F$. \\ 
	Similarly $(x \cap y) \to y_1 \in F$. Then by definition of a filter, $((x\cap y) \to x_1) \cap ((x \cap y) \to y_1) \in F$. \\ 
	Therefore $(x \cap y)\to (x_1 \cap y_1) \in F$, by Lemma \ref{lem2}. \\
	Similarly by changing the roles of $x, y$ with $x_1, y_1$ respectively we get $(x_1 \cap y_1)\to (x \cap y) \in F$. \\
	Thus $[x_1\cap y_1]= [x\cap y]$.
	
	\item Now we show that $L/F$ is closed under $\cup$.\\
	Let $x_1\in[x] , y_1\in[y]$. \\
	We shall show that $[x_1]\cup [y_1]= [x] \cup [y]$ i.e., $[x_1\cup y_1] = [x\cup y]$. Thus we have to show that
	$ (x \cup y) \to (x_1 \cup y_1) $ and $(x_1 \cup y_1) \to (x \cup y) $  both are in $F$. \\
	By assumption $x \to x_1, x_1 \to x, y \to y_1, y_1\to y \in F$. \\
	Now by Theorem \ref{lem1}.7, $x_1\to x \leq x_1 \to (x \cup y)$. Then $x_1 \to (x \cup y) \in F$. \\ 
	Similarly $y_1 \to (x \cup y) \in F$. Then by definition of a filter, $(x_1 \to (x \cup y)) \cap (y_1 \to (x \cup y)) \in F$. \\ 
	Therefore $(x_1 \cup y_1)\to (x \cup y) \in F$, by Lemma \ref{lem3}. \\
	Similarly by changing the roles of $x, y$ with $x_1, y_1$ respectively we get $(x \cup y)\to (x_1 \cup y_1) \in F$. \\
	Thus $[x_1\cup y_1]= [x\cup y]$.
	
	\item Now we show that $L/F$ is closed under $\to$.\\
	Let $x_1\in[x] , y_1\in[y]$. \\
	We shall show that $[x_1]\to [y_1]= [x] \to [y]$ i.e., $[x_1\to y_1] = [x\to y]$. Thus we have to show that
	$ (x \to y) \to (x_1 \to y_1) $ and ($x_1 \to y_1) \to (x \to y) $  both are in $F$. \\
	By assumption $x \to x_1, x_1 \to x, y \to y_1, y_1\to y \in F$. \\ 
	Then by definition of a filter, $(x_1 \to x)\ast(y \to y_1) \in F$. \\
	Now by Theorem \ref{lem1}.9, $x_1 \ast (x\to y)\ast(x_1 \to x)\ast(y \to y_1)\leq y_1$ \\
	By Residuation property, $(x_1 \to x)\ast(y \to y_1) \leq (x\to y)\ast (x_1\to y_1)$. Then $(x\to y)\ast (x_1\to y_1)\in F$. \\
	Similarly by changing the roles of $x, y$ with $x_1, y_1$ respectively we get $(x_1 \to y_1)\to (x \to y) \in F$. Thus $[x_1\to y_1]= [x\to y]$.
	
    \item Now we show that $L/F$ is closed under $\ast$.\\
	Let $x_1\in[x] , y_1\in[y]$. \\
	We shall show that $[x_1]\ast [y_1]= [x] \ast [y]$ i.e., $[x_1\ast y_1] = [x\ast y]$. Thus we have to show that
	$ (x \ast y) \to (x_1 \ast y_1) $ and $(x_1 \ast y_1) \to (x \ast y) $  both are in $F$.  \\
	By assumption $x \to x_1, x_1 \to x, y \to y_1, y_1\to y \in F$. \\ 
	Then by definition of a filter, $(x \to x_1)\ast(y \to y_1) \in F$. \\
	Now by Theorem \ref{lem1}.9, $x\ast y\ast(x \to x_1)\ast(y \to y_1)\leq x_1 \ast y_1$ \\
	By Residuation property, $(x \to x_1)\ast(y \to y_1)\leq (x\ast y) \to (x_1 \ast y_1)$. Then $(x\ast y) \to (x_1 \ast y_1) \in F$. \\  
	Similarly by changing the roles of $x, y$ with $x_1, y_1$ respectively we get $(x_1 \ast y_1)\to (x \ast y) \in F$. Thus $[x_1\ast y_1]= [x\ast y]$.
\end{itemize}
	So, we can see that $L/F$ is closed under the operations.  
\end{proof}
\begin{theorem}
$(L/F, \cap, \cup,\ast, \to ,[\bot], [1])$ is an IL-algebra with respect to the operations defined by \\
$[x]\cup [y]=[x\cup y]$ \\
$[x]\cap [y]=[x\cap y]$ \\
$[x]\ast [y]=[x\ast y]$ \\
$[x]\to [y]=[x\to y]$
\end{theorem}
\begin{proof} It is obvious that 
		$(L/F, \cap, \cup)$ is a bounded lattice with least element [$\bot$] and 
$(L/F, \ast, [1])$ is a commutative monoid with identity [1].

 Let $[x],[y],[z] \in L/F$ be such that $[x]\ast[y]\leq [z]$. So, $[x\ast y] \leq [z]$. Then by Proposition \ref{prop2}, $x\ast y \to z \in F$. Again by Theorem \ref{lem1}.8, $(x\ast y) \to z \leq x\to (y\to z)$. So we can derive that $x\to(y \to z) \in F$. Thus by Proposition \ref{prop2}, we get  $[x]\leq [y \to z]$. \\ 
Conversely, let $x\to (y\to z)\in F$. Now $(x\to (y\to z))\ast x\ast y\leq (y\to z)\ast y\leq z$, by Theorem \ref{lem1}.9. So, by residuation property, $x\to(y \to z)\leq (x\ast y)\to z$.\\
 Hence $(L/F, \cap, \cup,\ast, \to ,[\bot], [1])$ is an IL-algebra.
\end{proof}
\begin{theorem}
	If $F=\{x\in L:1\leq x\}$ then $[x]=\{x\}$.
\end{theorem}
\begin{proof}
	Let $x\in L$.\\
	Assume that $y\in [x].$ then $x\to y$ and $x\to y$ both are in $F$.\\
	Then by construction of $F$, $1\leq x\to y$ and $1\leq y\to x$. So, $x\leq y$ and $y \leq x$.\\
	Then by combining these two, we get $x=y$.
\end{proof}

\section{Special types of filters}
\begin{definition}
A filter $F$ of an IL-algebra $L$ is said to be distributive if\\
 $((x\cup y)\cap(x\cup z))\to (x\cup (y\cap z))\in F$ for all $x,y,z\in L$.
\end{definition}
\begin{example}
Let $F=\{1,b,c,d,\top\}$ in Example \ref{ex2}. Then $F$ is a distributive filter in $L$.
\end{example}

\begin{example} \label{ex3}
	Let, $X=\{ \bot,a,b,1,\top \}$. Lattice ordering, $\ast$ and $\rightarrow$ tables are the following	
	
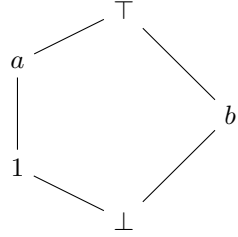
\begin{center}
		\begin{tikzpicture}[scale=.7]
		\node (top) at (0,2) {$\top$};
		\node (one) at (-2,-1) {$1$}; 
		\node (a) at (-2,1) {$a$};
		\node (b) at (2,0) {$b$}; 
		\node (bot) at (0,-2) {$\bot$};
		\draw (bot) -- (b) -- (top);
		\draw (bot) -- (one) -- (a) -- (top);
		\end{tikzpicture}
	\end{center}
\begin{center}
			\begin{tabular}{ l | l l l l l }
			\textbf{$\ast$} & \textbf{$\bot$} & \textbf{$a$} & \textbf{$b$} & \textbf{$1$} & \textbf{$\top$}  \\ \hline
			$\bot$ & $\bot$ & $\bot$ & $\bot$ & $\bot$ & $\bot$ \\ 
			$a$ & $\bot$ & $\top$ & $\top$ & $1$ & $\top$ \\ 
			$b$ & $\bot$ & $\top$ & $\top$ & $b$ & $\top$ \\ 
			$1$ & $\bot$ & $a$ & $b$ & $1$ & $\top$ \\ 
			$\top$ & $\bot$ & $\top$ & $\top$ & $\top$ & $\top$
		\end{tabular}		
	\hspace{1cm}
			\begin{tabular}{ l | l l l l l }
			\textbf{$\rightarrow$} & \textbf{$\bot$} & \textbf{$a$} & \textbf{$b$} & \textbf{$1$} & \textbf{$\top$}  \\ \hline
			$\bot$ & $\top$ & $\top$ & $\top$ & $\top$ & $\top$ \\ 
			$a$ & $\bot$ & $1$ & $\bot$ & $\bot$ & $\top$ \\
			$b$ & $\bot$ & $\bot$ & $1$ & $\bot$ & $\top$ \\
			$1$ & $\bot$ & $a$ & $b$ & $1$ & $\top$ \\ 
			$\top$ & $\bot$ & $\bot$ & $\bot$ & $\bot$ & $\top$
		\end{tabular}
	\end{center}		
	
	Then, $(L,\ast,\cup,\cap,\rightarrow,1,\top)$ is an IL-algebra. \\
    In this example, if we take $F=\{1,a,\top\}$ \\
    We can see, $((1\cup a)\cap(1\cup b))\to (1\cup (a\cap b)) = \bot$, which is not in $F$, then $F$ is not a distributive filter.
\end{example}
\begin{theorem}
If $F$ be a distributive filter in an IL-algebra $L$ then $L/F$ is distributive.
\end{theorem}
\begin{proof}
We know that $x\cup (y\cap z)\leq x\cup y, x\cup z$. Hence $x\cup (y\cap z)\leq (x\cup y)\cap (x\cup z)$ and so $x\cup (y\cap z)\to (x\cup y)\cap (x\cup z)\in F$. Therefore by Proposition \ref{prop2}, $[x]\cup ([y]\cap [z]) \leq ([x]\cup[y])\cap ([x]\cup [z])$. \\
On the other hand, from the definition of distributive filter and from Proposition \ref{prop2}, we have $([x]\cup[y])\cap ([x]\cup [z]) \leq [x]\cup ([y]\cap [z]) $. Hence $L/F$ is distributive.
\end{proof}
\begin{definition}
A filter $F$ of an IL-algebra $L$ is said to be prime if for any $x, y \in L, \ x\to y \in F$ or $y\to x \in F$.
\end{definition}
\begin{example}
	In Example  \ref{ex2}, if $F=\{b,c,d,1,\top\}$, then $F$ is a prime filter.
\end{example}
\begin{example}
    In Example \ref{ex1}, if we take $F=\{1,\top\}$ \\
    We can see, both $c\to d$ and $d\to c$ are not in $F$, then $F$ is not a prime filter. 
\end{example}
From the definition of prime filter and Proposition \ref{prop2}, it follows that
\begin{theorem}
If $F$ be a prime filter, then $L/F$ is linear.
\end{theorem}
\begin{definition}
	A filter $F$ of an IL-algebra $L$ is said to be maximal filter if it is not contained in other proper filters of $L$.
\end{definition}
\begin{example}\label{ex4}
	Let, $L=\{\bot,a,b,c,1,\top\}$. Lattice ordering, $*$ and $\rightarrow$ tables are the following \\
	
	\begin{center}
		\begin{tikzpicture}[scale=1]
		\node (top) at (0,5) {$\top$}; 
		\node (one) at (0,4) {$1$}; 
		\node (c) at (0,3) {$c$}; 
		\node (b) at (0,2) {$b$};
		\node (a) at (0,1) {$a$}; 
		\node (bot) at (0,0) {$\bot$};
		\draw (bot) --(a) -- (b) -- (c)--  (one) -- (top);
		\end{tikzpicture}
	\end{center}
\begin{center}
		\begin{tabular}{ l | l l l l l l }
			\textbf{$\ast$} & \textbf{$\bot$} & \textbf{$a$} & \textbf{$b$} & \textbf{$c$} & \textbf{$1$} & \textbf{$\top$}  \\ \hline
			$\bot$ & $\bot$ & $\bot$ & $\bot$ & $\bot$ & $\bot$ & $\bot$ \\ 
			$a$ & $\bot$ & $\top$ & $\top$ & $\top$ & $b$ & $\top$ \\ 
			$b$ & $\bot$ & $a$ & $b$ & $b$ & $b$ & $b$ \\ 
			$c$ & $\bot$ & $a$ & $b$ & $c$ & $c$ & $\top$ \\ 
			$1$ & $\bot$ & $a$ & $b$ & $c$ & $1$ & $\top$ \\ 
			$\top$ & $\bot$ & $a$ & $b$ & $\top$ & $\top$ & $\top$
		\end{tabular} 
		\hspace{0.5cm}	
		\begin{tabular}{ l | l l l l l l }
			\textbf{$\ast$} & \textbf{$\bot$} & \textbf{$a$} & \textbf{$b$} & \textbf{$c$} & \textbf{$1$} & \textbf{$\top$}  \\ \hline
			$\bot$ & $\top$ & $\top$ & $\top$ & $\top$ & $\top$ & $\top$ \\ 
			$a$ & $a$ & $\top$ & $\top$ & $\top$ & $\top$ & $\top$ \\
			$b$ & $\bot$ & $a$ & $\top$ & $\top$ & $\top$ & $\top$ \\
			$c$ & $\bot$ & $a$ & $b$ & $1$ & $1$ & $\top$ \\ 
			$1$ & $\bot$ & $a$ & $b$ & $c$ & $1$ & $\top$ \\ 
			$\top$ & $\bot$ & $a$ & $b$ & $b$ & $b$ & $\top$ 
		\end{tabular}

\end{center}
Then, $(L,\ast,\cup,\cap,\rightarrow,1,\top)$ is an IL-algebra. \\
In this example, $F_3=\{b,c,1,\top \}$ is a maximal filter. But, $F_1=\{1,\top \}$, $F_2=\{c,1,\top \}$ are not maximal filters, as both $F_1$ and $F_2$ are contained in the proper filter $F_3$.
\end{example}

\begin{example}\label{ex5}
	Let, $L=\{\bot,a,b,c,d,1,\top\}$. Lattice ordering, $*$ and $\rightarrow$ tables are the following \\
	
	\begin{center}
		\begin{tikzpicture}[scale=1]
		\node (top) at (0,2) {$\top$}; 
		\node (one) at (-1,1) {$1$};
		\node (a) at (-2,0) {$a$};
		\node (c) at (2,0) {$c$}; 
		\node (b) at (0,0) {$b$};
		\node (d) at (-1,-1) {$d$}; 
		\node (bot) at (0,-2) {$\bot$};
		\draw (bot) --(d) -- (a) --  (one) -- (top);
		\draw (d) -- (b) --  (one);
		\draw (bot) --(c) -- (top);
		\end{tikzpicture}
	\end{center}
\begin{center}
		\begin{tabular}{ l | l l l l l l l }
			\textbf{$\ast$} & \textbf{$\bot$} & \textbf{$a$} & \textbf{$b$} & \textbf{$c$} & \textbf{$d$} & \textbf{$1$} & \textbf{$\top$}  \\ \hline
			$\bot$ & $\bot$ & $\bot$ & $\bot$ & $\bot$ & $\bot$ &  $\bot$ & $\bot$ \\
			$a$ & $\bot$ & $a$ & $d$ & $c$ & $d$ & $a$ & $\top$ \\ 
			$b$ & $\bot$ & $d$ & $b$ & $c$ & $d$ & $b$ & $\top$ \\ 
			$c$ & $\bot$ & $c$ & $c$ & $d$ & $c$ & $c$ & $\top$ \\
			$d$ & $\bot$ & $d$ & $d$ & $c$ & $d$ & $d$ & $\top$ \\
			$1$ & $\bot$ & $a$ & $b$ & $c$ & $d$ & $1$ & $\top$ \\ 
			$\top$ & $\bot$ & $\top$ & $\top$ & $\top$ & $\top$ & $\top$ & $\top$
		\end{tabular} 
		\hspace{0.5cm}	
		\begin{tabular}{ l | l l l l l l l }
			\textbf{$\rightarrow$} & \textbf{$\bot$} & \textbf{$a$} & \textbf{$b$} & \textbf{$c$} & \textbf{$d$} & \textbf{$1$} & \textbf{$\top$}  \\ \hline
			$\bot$ & $\top$ & $\top$ & $\top$ & $\top$ & $\top$ & $\top$ & $\top$ \\ 
			$a$ & $\bot$ & $1$ & $b$ & $c$ & $b$ & $1$ & $\top$ \\
			$b$ & $\bot$ & $a$ & $1$ & $c$ & $a$ & $1$ & $\top$ \\
			$c$ & $\bot$ & $c$ & $c$ & $1$ & $c$ & $1$ & $\top$ \\
			$d$ & $\bot$ & $1$ & $1$ & $c$ & $1$ & $1$ & $\top$ \\
			$1$ & $\bot$ & $a$ & $b$ & $c$ & $d$ & $1$ & $\top$ \\ 
			$\top$ & $\bot$ & $\bot$ & $\bot$ & $\bot$ & $\bot$ & $\bot$ & $\top$ 
		\end{tabular}
\end{center}
Then, $(L,\ast,\cup,\cap,\rightarrow,1,\top)$ is an IL-algebra. \\
    In this example, $F_4=\{a,b,d,1,\top \}$ is a maximal filter. But, $F_1=\{1,\top \}$, $F_2=\{a,1,\top \}$, $F_3=\{b,1,\top \}$ are not maximal filters, as $F_1$, $F_2$ and $F_3$ are contained in the proper filter $F_4$.
\end{example}

\begin{definition}
Let $L$ be an IL-algebra. A non-empty subset $F$ of $L$ is said to be an implicative filter if 
\begin{itemize}
	\item $1\in F$ 
	\item If $x\to (y\to z) \in F$ and $x\to y \in F$ then $x\to z\in F$.
\end{itemize}
\end{definition}
\begin{example}
    In Example \ref{ex2}, $F=\{b,c,d,1,\top\}$ is an implicative filter.
\end{example}
\begin{example}
    In Example \ref{ex3}, if we take $F_3=\{b,c,1,\top\}$ \\
    We can see, $a\to (a\to \bot) \in F_3$ and $a\to \bot \in F_3$, but $a\to \bot \notin F_3$, then $F_3$ is not an implicative filter. \\
    Similarly, in Example \ref{ex5}, if we take $F_4=\{a,b,d,1,\top\}$ \\
    We can see, $c\to (c\to 1) \in F_4$ and $c\to c \in F_4$, but $c\to 1 \notin F_4$, then $F_4$ is not an implicative filter.
\end{example}
\begin{proposition}
If every element of an IL-algebra $L$ is idempotent and $F$ be a filter in $L$, then $F$ is an implicative filter.
\end{proposition}
\begin{proof}
Since $F$ is filter, $1 \in F$. Let for all $x,y,z \in L,\ x\to (y\to z)\in F$ and $x\to y\in F$. \\
Thus $(x\to (y\to z))\ast (x\to y)\in F$.\\
Now by Theorem \ref{lem1}.9, $x\ast (x\to y)\leq y$ and $x\ast (x\to (y\to z))\leq y\to z$. \\
Hence by Theorem \ref{lem1}.7 and \ref{lem1}.9, $x\ast x\ast(x\to y)\ast (x\to (y\to z))\leq y\ast (y\to z)\leq z$.\\
So, by residuation property, $(x\to y)\ast (x\to (y\to z))\leq x\to z$ (as $x\ast x=x$).\\
 Thus $x\to z \in F$.
\end{proof}
\textbf{Note:} Converse of the above theorem is not true. As we can see in example \ref{ex2}, $F=\{b,c,d,1,\top \}$ is an implicative filter. But, all elements of this IL-algebra $L$ are not idempotent.

\begin{definition}
	An algebraic structure  $L=(L,\cup ,\cap, 0, \to, \ast,1 )$ is said to be a residuated lattice if
\begin{itemize}
	\item $(L,\cup ,\cap, 0, 1)$ is a bounded lattice.
	\item $(L,\ast, 1)$ is a commutative monoid.	
	\item for any $x,y,z \in L$ , $x\ast y \leq z$ if and only if $x \leq y \to z$.
\end{itemize}
\end{definition}
\begin{proposition}\label{rl}
	In an IL-algebra $L$, $x\ast y\leq x$ if and only if $\top=1$ for all $x,y\in L$
\end{proposition}
\begin{proof}
	Let $x\ast y\leq x.$ Now $x=x\ast 1  \leq x\ast \top\leq x$. So $x\ast \top=x$. Thus $1\ast \top=1$ or $\top=1$. \\
	Conversely, assume that $\top=1$. Now $y\leq \top $ or $x\ast y\leq x\ast \top =x.$
\end{proof}
\begin{note}
	From the Proposition \ref{rl}, it follows that an IL-algebra $L$ satisfying $x\ast y\leq x$  for all $x,y\in L$ (or equivalently $\top=1$) is a residuated lattice. \\
	If we consider $\top=1$ in an IL-algebra, then by theorem \ref{lem1}.3, the condition `$x \cap y \in F$ for all $x,y\in F$' of the definition of filter $F$ on IL-algebra becomes redundant. Thus the concept of filter on IL-algebra generalizes the notion of filter on Residuated Lattice.
\end{note}
\begin{definition}
	A filter $F$ of an IL-algebra is called an affine filter if and only if $\top \to 1\in F$.
\end{definition}
\begin{example}
    In Example \ref{ex4}, $\top \to 1 = b$, so that makes $F_3=\{b,c,1,\top \}$ an affine filter. \\
    Whereas, both $F_1=\{1,\top \}$ and $F_2=\{c,1,\top \}$ are not affine filters.
\end{example}
\begin{theorem}
	Let $F$ be an affine  filter in an  IL-algebra $L$, then $L/F$ is a residuated lattice.
\end{theorem}
\begin{proof}
	In any IL-algebra $1\leq \top$ implies $1\to \top \in F$. Again by definition of affine filter $ \top \to 1\in F$, then by definition of equivalence classes defined earlier, $[1]=[\top]$. \\
	Hence $L/F$ is a residuated lattice.
\end{proof}
 
 \section{Conclusion}
We introduce the concept of filters on IL-algebra in this paper. Some other algebraic structures closely related to IL-algebra, namely ILZ-algebra, CL-algebra, are not considered here. Filters on these structures may be taken into consideration in future. Further properties on filters of IL-algebra may also be investigated. This paper is a first step in this new area of research. It will open scope of work on multiple areas, study of filters on algebraic structures related with linear logic, reflection of algebraic results based on filters in corresponding logic and many more.
\section{Acknowledgments}

The last two authors acknowledge Department of Higher Education, Science \& Technology and Bio-Technology, Government of West Bengal, India for the financial support in the research project 257(Sanc)/ST/P/S\&T/16G-45/2017 dated 25.03.2018.



\begin{thebibliography}{00}
\bibitem{ras} H. Rasiowa, {\it An Algebraic Approach to Non-Classical Logics}, North Holland, Amsterdam, 1974.
\bibitem{haj} P. H\'ajek, {\it Metamathematics of Fuzzy Logic} Kluwer Academic Publishers, Dordrecht, 1998.
\bibitem{tur1} E. Turunen, {\it BL-algebras of basic fuzzy logic}, Mathware and Soft Computing 6 (1999) 49–61.
\bibitem{tur2} E. Turunen, {\it Boolean deductive system of  BL-algebras}, Archive for Mathematical Logic 40(6) (2001) 467–473.
\bibitem{sae1} S. Motamed, L. Torkzadeh, A. Borumand Saeid and N. Mohtashamnia, {\it Radical of filters in  BL-algebras}, Mathematical Logic Quarterly 57(2) (2011) 166–179.
\bibitem{sae2} A. Borumand Saeid and S. Motamed, {\it Normal filters in BL-algebras}, World Applied Science Journal 7 (2009) 70–76.
\bibitem{sae3} A. Borumand Saeid and S. Motamed, {\it A new filter in BL-algebras}, Journal of Intelligent and Fuzzy Systems 27(6) (2014) 2949–2957.
\bibitem{sae4} M. Haveshki, A. Borumand Saeid and E. Eslami, {\it Some types of filters in  BL-algebras}, Soft Computing 10(8) (2006) 657–664. 
\bibitem{mog} S. Motamed and J. Moghaderi, {\it Primary Filters in BL-algebras},  New Mathematics and Natural ComputationVol. 15, No. 03, pp. 447-461 (2019). 
\bibitem{mot} S. Motamed and L. Torkzadeh, {\it A new class of  BL-algebras}, Soft Computing 21(3) (2017) 687–698. 
\bibitem{kon} M. Kondo and W. A. Dudek, {\it Filter theory of BL algebras}, Soft Computing (2008) 12:419–423.
\bibitem{bor} R.A. Borzooei, S. Khosravi Shoar, R. Ameri, {\it Some types of filters in MTL-algebras}, Fuzzy Sets and Systems 187, 92 - 102. (2012).
\bibitem{tro} A. S. Troelstra, {\it Lectures on linear logic} No. 29, Center for the Study of Language and Information, Stanford, 1992.
\bibitem{gir} J-Y Girard, {\it Linear logic}, Theoretical Computer Science, 50(1), 1987,1-101.
\bibitem{ono}   N.  Galatos,  P.  Jipsen,  T.  Kowalski,  and  H.  Ono, {\it Residuated lattices: an algebraic glimpse atsubstructural logics} volume 151.  Elsevier, (2007).
\bibitem{senth} J. Sen, {\it Some Embeddings In Linear Logic And Related Issues}, Ph.D. Thesis, University of Calcutta, India, 2001. 
\bibitem{sen} M.K. Chakraborty and J. Sen, MV-algebra embedded in a CL-algebra, in {\it International Journal of Approximate Reasoning}, 18, pp. 217-229, (1998).


\end{thebibliography}
\end{document}